\documentclass[a4paper]{article}
\usepackage{graphicx,subcaption}
\usepackage{mathtools}
\usepackage{amssymb}
\usepackage{amsmath}
\usepackage{amsthm}
\usepackage{xcolor,soul}
\usepackage{fancyhdr}
\usepackage{epstopdf}
\usepackage{hyperref}
\usepackage{soul}
\usepackage{comment}
\usepackage{tabularx}
\usepackage{booktabs}
\usepackage{lineno}

\newtheorem{theorem}{Theorem}[section]

\newtheorem{lemma}[theorem]{Lemma}

\newtheorem{proposition}[theorem]{Proposition}
\newtheorem{remark}[theorem]{Remark}

\usepackage[]{geometry}
\usepackage{tikz}

\usetikzlibrary{positioning}
\usetikzlibrary{arrows}
\usetikzlibrary{arrows.meta}
\usetikzlibrary{calc}
\usetikzlibrary{shapes}

\newcolumntype{C}{>{\centering\arraybackslash}X}

\graphicspath{ {./figures/} }

\title{Traveling fronts in a spatial epidemic model \\ with slow loss of immunity}
\author{Rossella Della Marca$^{1}$, Gabriele Grifò$^{2,3}$, Annalisa Iuorio$^{4}$, Mattia Sensi$^{5,\ast}$\\
{\scriptsize $^{1}$Department of Mathematics and Applications, University of Naples ``Federico II", via Cintia, I-80126 Naples, Italy}\\
{\scriptsize $^{2}$Istituto Nazionale di Alta Matematica ``F. Severi", Piazzale Aldo Moro 5, I-00185 Rome, Italy} \\
{\scriptsize $^{3}$Department of Mathematics and Computer Science, University of Palermo, Via Archirafi 34, I-90123 Palermo, Italy} \\
{\scriptsize $^{4}$Department of Engineering, Parthenope University of Naples, Centro Direzionale - Isola C4, 80143 Naples, Italy}\\
{\scriptsize $^{5}$Department of Mathematics, University of Trento, Via Sommarive 14, I-38123 Trento, Italy}\\
{\footnotesize $^{\ast}$Corresponding author: \texttt{mattia.sensi@unitn.it}}}
\date{\today}

\begin{document}

\maketitle

\begin{abstract}
We investigate the emergence of traveling front solutions in a spatial SIRS epidemic model with diffusion acting on the infected population. The model exhibits a natural slow-fast structure due to the presence of a small parameter governing the loss of immunity, which induces a separation of scales in the dynamics. Using a traveling wave reduction, the PDE system is transformed into a singularly perturbed system of ODEs, which we analyze within the framework of Geometric Singular Perturbation Theory. In the singular limits, we study the fast excursions governed by the layer problem, and the slow evolution close to the critical manifold. In particular, we identify an entry-exit mechanism tracking the transitions between slow and fast regimes, and derive a quantitative characterization of the entry-exit dynamics. Numerical simulations of the full system confirm the validity of the proposed geometric picture. The traveling front is shown to consist of a concatenation of local, fast, and slow segments, in agreement with the theoretical analysis. 
\\

\end{abstract}

\section{Introduction}

Compartmental models have been a cornerstone of mathematical epidemiology since the work of Kermack and McKendrick \cite{kermack1927contribution}, describing how an infectious disease spreads by partitioning a well-mixed population into a few classes (typically, in one of its most common formulations, Susceptible, Infected and Recovered individuals) linked by a system of ordinary differential equations (ODEs) describing the flow of individuals between compartments. In the Susceptible-Infected-Recovered-Susceptible (SIRS) variant, recovered individuals eventually lose immunity and return to the susceptible class, which makes it a natural framework for diseases that do not confer permanent protection and introduces a mechanism for recurrent outbreaks.

A common feature of such systems is that the rates governing the different pathways act on widely separated time scales: the infectious period is often much shorter than the average duration of acquired immunity, so that the loss of immunity is a genuinely \emph{slow} process compared to the \emph{fast} infection and recovery dynamics. This scale separation endows the equations with a singularly perturbed structure that can be systematically exploited through Geometric Singular Perturbation Theory (GSPT) \cite{fenichel1979geometric,hek2010geometric,kuehn2015multiple,wechselberger2020geometric}, which decomposes the dynamics into a \emph{layer} (fast) problem and a \emph{reduced} (slow) problem organized around a lower-dimensional \emph{critical manifold}. This approach has proven fruitful for a range of epidemic models, including fast-slow SIR, SIRS and SIRWS systems \cite{jardon2021geometric}, SIRS models on homogeneous networks \cite{jardon2021geometric2}, SIRS models with secondary infections \cite{kaklamanos2024geometric}, and related systems in mathematical biology \cite{borsotti2026geometric}.

Well-mixed ODE models capture the temporal evolution of an epidemic, yet they fail to account for its \emph{spatial} spread; this component is crucial, for example, in the spread of avian flu, which is periodically re-introduced in farms by migratory birds that carry the disease with mild or non-existent symptoms \cite{bourouiba2011interaction,liu2008spread,rao2008modeling,si2009spatio,vaidya2012avian}. Since individual mobility plays a decisive role in propagation, spatially extended models obtained by adding diffusion terms to compartmental systems have received considerable attention, a recurring theme being the emergence of \emph{traveling waves} connecting an endemic state behind the front to a disease-free state ahead of it. These have been studied from several angles: existence and critical speed of periodic traveling waves via fixed-point and sub/super-solution arguments \cite{wu2021periodic}; Turing-type patterns and, in cross-diffusion settings, heterogeneous steady states and spikes \cite{ahmadpoortorkamani2025spatiotemporal,chang2022sparse,li2025positive}; asymptotic localization of endemic profiles in the small-diffusion regime \cite{peng2024spatial,peng2025varying}; and data-driven mobility approaches for specific outbreaks \cite{parra2025challenges}. Of particular relevance here are models whose fast component is not the host population itself: infected individuals on a line of fast diffusion within a two-dimensional domain \cite{berestycki2020propagation}, which markedly enhances the spreading speed even near the epidemic threshold, or a free-living pathogen diffusing in the environment \cite{pang2019sis}, also studied via GSPT.

Comparatively fewer works combine the spatial (traveling-wave) and temporal (fast-slow) viewpoints, analyzing the sharp fronts of a diffusive epidemic model directly through GSPT, in the spirit of GSPT analyses of shock-fronted traveling waves \cite{li2021shock} and related multiple-time-scale wave problems \cite{popovic2025burgers}. This is the perspective adopted here. We study a spatial SIRS model in which diffusion acts only on the infected compartment and immunity is lost on a slow time scale, controlled by a small parameter. Numerical simulations show that the model supports invasion fronts traveling at constant speed, whose profiles display a clear separation of scales: sharp transition layers in the infected component coexist with slow variations in the susceptible population. Introducing a traveling-wave coordinate reduces the partial differential equations (PDEs) to a singularly perturbed system of ODEs in non-standard form, which we analyze with GSPT: we characterize the local dynamics near the disease-free and endemic equilibria, describe the fast excursions through the layer problem and an associated invariant, and study the slow evolution near the critical manifold. We identify an \emph{entry-exit} mechanism \cite{de2016entry,kaklamanos2025entry,schecter2008exchange1,schecter2008exchange2} governing the transitions between slow and fast regimes, and derive a quantitative characterization of the corresponding \emph{exit time}. The analytical picture is validated against direct numerical simulations of the full PDE system, showing excellent agreement. From a modeling standpoint, the interplay between slow and fast dynamics translates into alternating phases of gradual adaptation of the susceptible pool and rapid outbreaks, with the immunity-loss rate setting both the duration of the quiescent phases and the spatial separation between successive outbreaks.

The remainder of the paper is organized as follows. Section~\ref{sec:model} introduces the SIRS PDE model, its non-dimensional form and preliminary numerical observations. Section~\ref{sec:analysis} contains the core of the analysis: after reducing the PDE to a traveling-wave ODE system, we study the local dynamics near the equilibria (Section~\ref{sec:lin_stab}), the fast/layer dynamics (Section~\ref{sec:fast_lim}), the slow flow and the entry-exit mechanism (Section~\ref{sec:slow_flow}), and we validate the resulting geometric picture numerically (Section~\ref{sec:num_valid}). Finally, Section~\ref{sec:discussion} discusses our findings, compares them with the planar SIRS model of \cite{jardon2021geometric}, and outlines directions for future work.

\section{PDE model}
\label{sec:model}

The SIRS structure is the minimal compartmental setting in which immunity is only \emph{temporary}: after recovery, individuals return to the susceptible class at a finite rate, so that the disease is not permanently cleared and the population can sustain recurrent outbreaks. This makes the SIRS model a natural framework for infections that do not confer lifelong protection, and it is precisely the slow return of recovered individuals to susceptibility that will supply the small parameter driving our multiscale analysis. For simplicity, we neglect vital dynamics (births and deaths): on the time scales of interest the demographic turnover is much slower than the epidemiological processes, and omitting it avoids introducing further small parameters without affecting the qualitative fast-slow structure we wish to describe. 

As for the spatial coupling, we let diffusion act on the infected compartment only. The rationale is twofold. On the modeling side, in many outbreaks it is the movement of infectious individuals (rather than of the susceptible or recovered background) that predominantly drives the geographic advance of the epidemic, so that concentrating mobility on the $I$ compartment isolates the mechanism most directly responsible for spatial invasion. On the analytical side, diffusing a single compartment keeps the traveling-wave reduction tractable while still producing the sharp fronts we are interested in. Moreover, it follows a classical modeling tradition: in the fox-rabies models of Murray, Stanley and Brown \cite{murray1986spatial} and of K\"all\'en \cite{kallen1984thresholds} (see also \cite{kot2001elements}), the spatial spread of the disease is due to the diffusion of infectives alone, so that mobility is genuinely a property of the infected class; the same choice is made in the benchmark SIR system of \cite{berestycki2020propagation}, where the susceptible population is assumed not to move. We stress that this is a simplifying assumption: incorporating diffusion of the susceptible compartment, or distinct diffusivities across compartments, is a natural direction for future work (see Section~\ref{sec:discussion}).

We start from a classical SIRS (Susceptible-Infected-Recovered-Susceptible) compartmental model without demography, described by the following system of ODEs:
\begin{equation}
\label{eq:sis}
    \begin{cases}
    S'=-\beta SI +\tilde{\delta} R, \\
    I'=\beta SI -\gamma I,\\
    R'=\gamma I -\tilde{\delta} R.
    \end{cases}   
\end{equation}
Here, the total population $N$ is partitioned with respect to an ongoing epidemic into Susceptible ($S$), Infected ($I$) and Recovered ($R$) individuals. The parameters are: $\beta$, aggregated parameter representing contact rate and infectivity of the virus; $\gamma$, recovery rate, which is the inverse of the average infectivity window duration; $\tilde{\delta}$, loss of immunity rate. We choose mass-action incidence $\beta SI$, in line with the model of rabies in foxes described in \cite[p. 332]{kot2001elements}. 
Note that, without any loss of generality, we normalize the total population density by assuming $N=S+I+R=1$. Indeed, any constant population level $N_0>0$ can be reduced to this case by applying an {\it ad-hoc} rescaling to the compartment densities.

With the aforementioned motivation and references in mind, system \eqref{eq:sis} becomes
\begin{equation}
\begin{cases}
    \dfrac{\partial S}{\partial \tilde{t}}=-\beta SI +\tilde{\delta} R, \medskip \\
    \dfrac{\partial I}{\partial \tilde{t}}=\beta SI -\gamma I + \tilde{D} \dfrac{\partial^2 I}{\partial \tilde x^2},\medskip\\
    \dfrac{\partial R}{\partial \tilde{t}}=\gamma I - \tilde{\delta}R.
\end{cases}   
\label{eq:SIRS_red_sp}
\end{equation}
and its nondimensional version reduces to
\begin{equation}
\begin{cases}
    \dfrac{\partial S}{\partial t}=-\mathcal{R}_0 SI +\delta R, \medskip \\
    \dfrac{\partial I}{\partial t}=\mathcal{R}_0 SI - I + \dfrac{\partial^2 I}{\partial x^2},\medskip\\
    \dfrac{\partial R}{\partial t }=I - \delta R,
\end{cases}   
\label{eq:SIRS_red_sp_adim}
\end{equation}
where $t=\gamma\tilde{t}$, $x=\sqrt{\frac{\gamma}{\tilde{D}}}\tilde{x}$, $\delta=\frac{\tilde{\delta}}{\gamma}$, and $0<\delta\ll 1$; i.e. we assume that loss of immunity happens at a much slower rate than the infection dynamics \cite{borsotti2026geometric,jardon2021geometric,jardon2021geometric2,kaklamanos2024geometric}. In the following let us focus on the case $\mathcal{R}_0:=\beta/\gamma>1$, as otherwise the disease will simply die out asymptotically.

\subsection{Preliminary numerical observations}\label{sec:prelim}

Let us consider the spatially homogeneous equilibria ($\partial^2 I/\partial x^2=0$) of system \eqref{eq:SIRS_red_sp_adim}, which are given by 
\begin{equation} 
\label{eq:equil}
    [\text{DFE}]: \;(S^\ast_1, I^\ast_1,R^*_1) = (1,0,0), \quad [\text{EE}] : \;(S^\ast_2, I^\ast_2, R^*_2) = \left( \frac{1}{\mathcal{R}_0}, \frac{\delta}{1+\delta} \left( \frac{\mathcal{R}_0 - 1}{\mathcal{R}_0} \right), \frac{1}{1+\delta} \left( \frac{\mathcal{R}_0 - 1}{\mathcal{R}_0} \right) \right).
\end{equation}

\noindent It should be noted that the first equilibrium corresponds to the disease-free state (DFE), while the second represents an endemic state (EE) in which the infection persists in the population. From a biological perspective, in the regime $\mathcal{R}_0 > 1$, the stability properties of the DFE and EE in the spatially homogeneous dynamics suggest that invasion phenomena may occur once spatial effects are taken into account. In particular, one expects that a localized infection introduced in a susceptible population may spread through the domain, leading to the formation of invasion fronts connecting the DFE to an infected region in which the direction of propagation depends on the sign of the wave speed.

To investigate such a scenario, we perform numerical simulations of system \eqref{eq:SIRS_red_sp_adim} starting from localized perturbations of the DFE and the results are shown in Figure~\ref{fig:numerical_simulations}. Panel \ref{fig:numerical_simulations}(a) displays the spatial profiles of $S(x,t_{\text{end}})$, $I(x,t_{\text{end}})$ and $R(x,t_{\text{end}})$, showing the presence of a sharp transition layer between the region close to the EE and the region where the infection is absent. This panel indicates the emergence of an invasion front. Moreover, the temporal evolution of the wave speed, reported in panel \ref{fig:numerical_simulations}(b), reveals that after an initial transient period the propagation speed converges to a constant value, suggesting that the observed front travels with constant velocity. This is further confirmed by the density plots presented in panels \ref{fig:numerical_simulations}(c) and \ref{fig:numerical_simulations}(d), where the interface between susceptible and infected regions follows approximately linear trajectories in the $(t,x)$ plane, providing clear evidence of constant-speed propagation. Overall, these observations indicate that system \eqref{eq:SIRS_red_sp_adim} supports traveling front solutions of the form \[ (S,I, R)(x,t) = (S, I, R)(x - ct), \] connecting the EE to the DFE, where $c \in \mathbb{R}^+$  is the positive migration speed. Moreover, the numerical profiles suggest the presence of a clear separation of spatial scales, with a sharp transition in the infected component and a slower variation in the susceptible population. This structure hints at an underlying slow-fast mechanism, which will be exploited in the subsequent analysis using GSPT.

\begin{figure}[t!]
    \centering
    \includegraphics[width=1\linewidth]{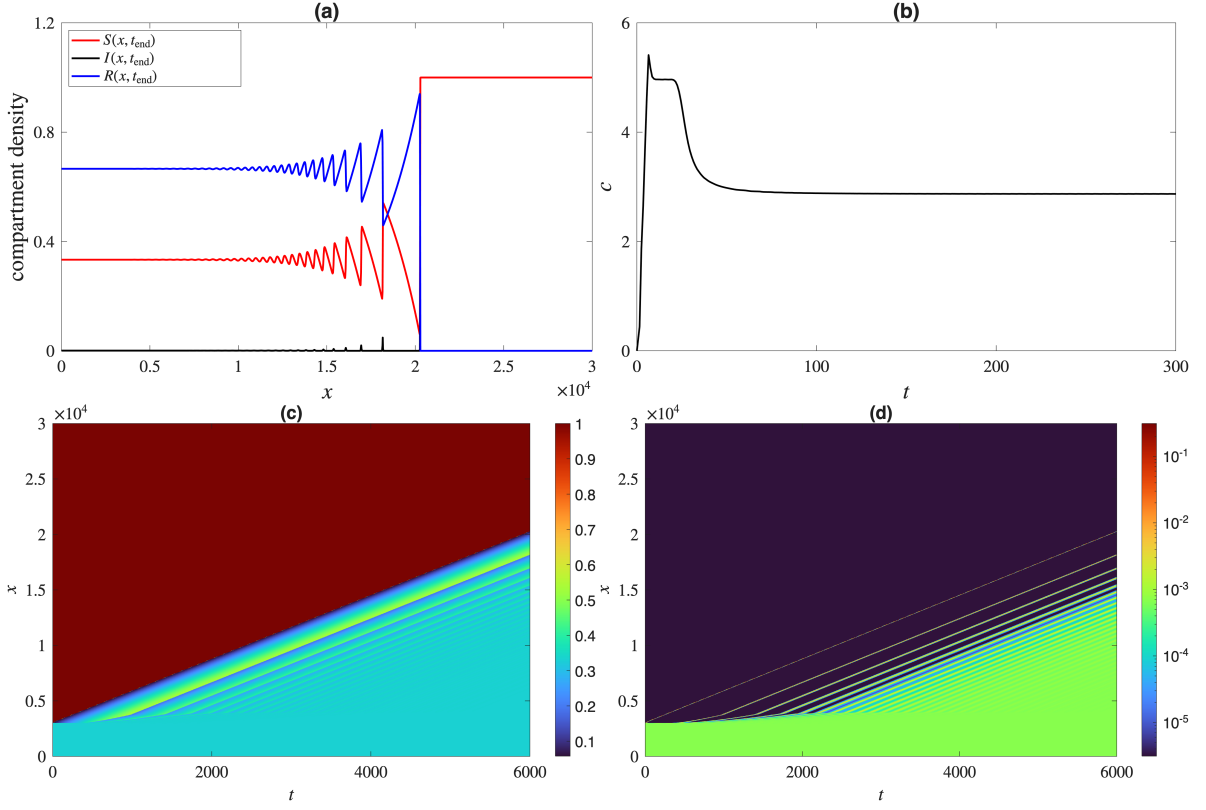}
    \caption{Numerical simulations of system \eqref{eq:SIRS_red_sp_adim} illustrating the emergence of traveling front solutions. Parameter values: $\mathcal{R}_0=3$ and $\delta=10^{-3}$.
    \textbf{(a)} Spatial profiles of the susceptible ($S$, red), infected ($I$, black) and recovered ($R$, blue) populations at a fixed final time, showing the formation of a sharp invasion front. \textbf{(b)} Temporal evolution of the wave speed $c(t)$, converging to the constant value $c_{\text{crit}}$ \eqref{eq:vel_crit}. \textbf{(c)}, \textbf{(d)} Space-time plots of $S$ and $I$ (the latter in logarithmic scale, for visual clarity), respectively. \label{fig:numerical_simulations}}
\end{figure}

In what follows, we will reduce system \eqref{eq:SIRS_red_sp_adim} to a system of ODEs through the introduction of a traveling wave coordinate, and study it with techniques from GSPT.

\section{Existence of traveling fronts}
\label{sec:analysis}

Let us consider system \eqref{eq:SIRS_red_sp_adim} and introduce the traveling wave coordinate $\hat{\xi}=x-ct$, where $c\in\mathbb{R}$ denotes the (\emph{a priori} unknown) migration speed of the front. Then,
$$
\dfrac{\partial^2 I}{\partial x^2}=
\dfrac{{\rm d}^2 I}{{\rm d} \hat{\xi}^2}, \quad 
\dfrac{\partial S}{\partial t}=-c\dfrac{{\rm d} S}{{\rm d} \hat{\xi}},\quad 
\dfrac{\partial I}{\partial t}=-c\dfrac{{\rm d} I}{{\rm d} \hat{\xi}},\quad \dfrac{\partial R}{\partial t}=-c\dfrac{{\rm d} R}{{\rm d} \hat{\xi}}.$$
Moreover, we introduce the auxiliary variable $J=\frac{{\rm d} I}{{\rm d} \hat{\xi}}$, so that system \eqref{eq:SIRS_red_sp_adim} becomes a system of first-order ODEs, namely
\begin{equation}
\begin{cases}
    \dfrac{{\rm d} S}{{\rm d} \hat{\xi}}=\dfrac{1}{c} \left(\mathcal{R}_0 SI -\delta R \right),\medskip \\
    \dfrac{{\rm d} I}{{\rm d} \hat{\xi}}=J, \medskip\\
    \dfrac{{\rm d} J}{{\rm d} \hat{\xi}}= -cJ-\mathcal{R}_0 SI+I,\medskip\\
    \dfrac{{\rm d} R}{{\rm d} \hat{\xi}}=\dfrac{1}{c} \left(-I+\delta R \right).
\end{cases}   
\label{eq:PTW_SIRS_interm}
\end{equation}
We note that, in system \eqref{eq:PTW_SIRS_interm}, 
$$
\dfrac{{\rm d} }{{\rm d} \hat{\xi}} \left( c(S+I+R) + J \right) =0,
$$
hence we may reduce its dimensionality by writing $R=K-S-I-J/c$, where $K \in \mathbb{R}$ is a constant which depends on the initial conditions of the system. Considering that we are investigating fronts connecting the EE and the DFE \eqref{eq:equil}, we evaluate the above condition at the DFE (or at the EE, with slightly less immediate calculations) to obtain $K=1$, so that system \eqref{eq:PTW_SIRS_interm} becomes
\begin{equation}
\begin{cases}
    \dfrac{{\rm d} S}{{\rm d} \hat{\xi}}=\dfrac{1}{c} \left[\mathcal{R}_0 SI -\delta(1-S-I)+\frac{\delta}{c}J \right],\medskip \\
    \dfrac{{\rm d} I}{{\rm d} \hat{\xi}}=J, \medskip\\
    \dfrac{{\rm d} J}{{\rm d} \hat{\xi}}=-cJ-\mathcal{R}_0 SI+I.
\end{cases}   
\label{eq:PTW_SIRS_red_sp}
\end{equation}
Recall that we assumed $0<\delta \ll 1$, whereas $c \in \mathcal{O}(1)$ and $1<\mathcal{R}_0\in \mathcal{O}(1)$; system \eqref{eq:PTW_SIRS_red_sp} is then a singularly perturbed system in non-standard form \cite{wechselberger2020geometric}.

To facilitate the forthcoming analysis, we reverse the traveling-wave coordinate by setting $\xi=-\hat{\xi}$. System \eqref{eq:PTW_SIRS_red_sp} then becomes
\begin{equation}
\begin{cases}
    \dfrac{{\rm d} S}{{\rm d} \xi}=\dfrac{1}{c} \left[\delta(1-S-I) -\mathcal{R}_0 SI - \frac{\delta}{c}J \right], \medskip \\
    \dfrac{{\rm d} I}{{\rm d} \xi}=-J,\medskip \\
    \dfrac{{\rm d} J}{{\rm d} \xi}= cJ + \mathcal{R}_0 SI-I .
\end{cases}   
\label{eq:PTW_SIRS_red_true}
\end{equation}
We note that this change of variables reverses the orientation of the traveling-wave orbit; hence, when returning to the original coordinate, the orbit must be traversed in the opposite direction. Throughout the manuscript, we compare systems \eqref{eq:PTW_SIRS_red_sp} and \eqref{eq:PTW_SIRS_red_true} to highlight the additional difficulties arising when considering the former. The singular parameter here is the immunity-loss rate $\delta$: setting $\delta=0$ decouples the fast infection-recovery dynamics from the slow drift of the susceptible pool. Following the GSPT program, we shall analyze the two singular limits separately: the \emph{layer problem} (Section~\ref{sec:fast_lim}), obtained as $\delta\to0$, and the \emph{reduced problem} on the critical manifold (Section~\ref{sec:slow_flow}), and finally combine the resulting fast and slow segments into a global description of the front, valid for small $\delta>0$.

The biologically feasible region for system \eqref{eq:PTW_SIRS_red_true} is
\begin{equation}
\label{eq:biol_feas}
\Delta := \left\{ (S,I,J) \in \mathbb{R}^3 \; | \; 0\leq S,I\leq 1,\; S+I+\frac{J}{c} \leq 1\right\}.
\end{equation}
However, note that this region is \emph{not} necessarily forward invariant; it is easy to see that, for example, the derivative of $I$ evaluated at $I=0$ can in principle have any sign.

For later reference, note that the $S$-nullcline of system \eqref{eq:PTW_SIRS_red_sp}, i.e.\ the set on which $\mathrm{d}S/\mathrm{d}\xi=0$, is given by 
\begin{equation}
    I=\dfrac{\delta(1-S-\frac{1}{c}J )}{\mathcal{R}_0 S+\delta},
    \label{eq:S_nullcline}
\end{equation}
which, for any given value of $J$, is a branch of a hyperbola with vertical asymptote at $S=-\delta/\mathcal{R}_0$ and horizontal asymptote at $I=-\delta/\mathcal{R}_0$; both lie outside the biologically feasible region $\Delta$ \eqref{eq:biol_feas} for $\delta,\mathcal{R}_0>0$.

Before we begin our fast-slow analysis of the model \eqref{eq:PTW_SIRS_red_true}, we focus on the local nature of its two equilibria, which are given by
\begin{equation} 
\label{eq:equil_con_J}
    [\text{DFE}]: \quad (S^\ast_1, I^\ast_1,J^*_1) = (1,0,0), \qquad [\text{EE}] : \quad (S^\ast_2, I^\ast_2, J^*_2) = \left( \frac{1}{\mathcal{R}_0}, \frac{\delta}{1+\delta} \left( \frac{\mathcal{R}_0 - 1}{\mathcal{R}_0} \right) ,0\right).
\end{equation}

Capturing the qualitative behavior observed in Section \ref{sec:prelim} via  Figure \ref{fig:numerical_simulations}, the traveling fronts considered here exhibit a clear separation of scales, with sharp transition layers and extended regions of gradual variation. This observation naturally suggests that the underlying dynamics can be decomposed into distinct regimes, corresponding to local behavior near the equilibria, fast excursions governed by the layer problem, and slow evolution along the critical manifold. This decomposition provides a natural guideline for the analysis carried out in the following subsections.

\subsection{Local dynamics near the equilibria}
\label{sec:lin_stab}

The Jacobian of \eqref{eq:PTW_SIRS_red_true} at a generic point $(S,I,J)$ is given by 
\begin{equation}
    \label{eq:jacob_full}
\mathcal{J}=\begin{pmatrix}
-\dfrac{\mathcal{R}_0 I +\delta}{c} & -\dfrac{\mathcal{R}_0 S +\delta}{c} & - \dfrac{\delta}{c^2} \medskip \\
0 & 0 & -1 \medskip \\
\mathcal{R}_0 I & \mathcal{R}_0 S-1  & c
\end{pmatrix}.
\end{equation}
We explore the local stability of the DFE and the EE \eqref{eq:equil_con_J} in the two following lemmas. In the following, we restrict our attention to traveling wave fronts with positive migration speed, meaning $c>0$, since our goal is to characterize the class of invasion fronts moving to the right which we have observed numerically in Figure \ref{fig:numerical_simulations}.
\begin{lemma}\label{lem:local_stab_DFE}
Assume $\mathcal{R}_0>1$ and $c>0$. Then, for the DFE of system \eqref{eq:PTW_SIRS_red_true}, one eigenvalue is real, negative and $\mathcal{O}(\delta)$, whereas two eigenvalues of $\mathcal{J}_1^\ast$ have positive $\mathcal{O}(1)$ real part; in particular, the eigenvalues are real if $c^2\geq 4(\mathcal{R}_0-1)$, and complex conjugate otherwise.
\end{lemma}
\begin{proof}Evaluating \eqref{eq:jacob_full} at the DFE \eqref{eq:equil_con_J}, we obtain
$$
\mathcal{J}^*_1=\begin{pmatrix}
-\dfrac{\delta}{c} & -\dfrac{\mathcal{R}_0+\delta}{c} & - \dfrac{\delta}{c^2} \medskip \\
0 & 0 & -1 \medskip \\
0 & \mathcal{R}_0 -1 & c
\end{pmatrix},
$$
which has eigenvalues $\lambda_0=-\frac{\delta}{c}$ (associated to the $S$-axis, on which the DFE is slowly attracting) and
$$
\lambda_{1,2}=\dfrac{c\pm \sqrt{c^2-4(\mathcal{R}_0-1)}}{2}.
$$
The real part of both is positive, and the eigenvalues are real and positive if $c^2\geq 2\sqrt{\mathcal{R}_0-1}$, and complex conjugate otherwise
\end{proof}
The threshold value $c=2\sqrt{\mathcal{R}_0-1}$ is precisely the one identified by our numerical integration; for ease of referencing, we introduce the notation
\begin{equation}\label{eq:vel_crit}
    c_{\text{crit}}=2\sqrt{\mathcal{R}_0-1}.
\end{equation}
\begin{remark}
Observing the numerical behavior of the system (recall Figure \ref{fig:numerical_simulations}), we expect to be in the case of real eigenvalues; indeed, in our simulation we obtained $c\approx 2.87$, with the threshold value being $\sqrt{8}\approx 2.83$. Note that a similar condition on the parameters $c$ and $\mathcal{R}_0$ will play a role in Section \ref{sec:slow_flow}, where we explore the slow dynamics of system \eqref{eq:PTW_SIRS_red_sp}. However, this seems to be a biological constraint rather than a mathematical one: what stops us from leaving the DFE through amplified oscillations is the necessity of traveling through negative values of $I$, the fraction of infected individuals in the population, which is clearly not biologically relevant.
\end{remark}

\begin{lemma}\label{lem:local_stab_EE}
Assume $\mathcal{R}_0>1$ and $c>0$. Then, there exists $\delta_0>0$ such that, for every $0<\delta<\delta_0$, the Jacobian $\mathcal{J}_2^\ast$ at the EE has one real, positive, $\mathcal{O}(1)$ eigenvalue, and a complex-conjugate pair with negative $\mathcal{O}(\delta)$ real part.
\end{lemma}
\begin{proof}Evaluating \eqref{eq:jacob_full} at the EE \eqref{eq:equil_con_J}, we obtain
$$
\mathcal{J}^*_2=\begin{pmatrix}
-\dfrac{\mathcal{R}_0 I^*_2+\delta}{c} & -\dfrac{1+\delta}{c} & - \dfrac{\delta}{c^2} \medskip\\
0 & 0 & -1 \medskip \\
\mathcal{R}_0 I^*_2 & 0 & c
\end{pmatrix}.
$$

Recall that $I_2^*\in \mathcal{O}(\delta)$; the determinant is given by $0<\frac{(1+\delta)\mathcal{R}_0I^*_2}{c}\in \mathcal{O}(\delta)$, whereas the trace is $0<c-\frac{\mathcal{R}_0I^*_2+\delta}{c}\in \mathcal{O}(1)$. Hence, the product of the three eigenvalues is real, positive and $\mathcal{O}(\delta)$, whereas their sum is real, positive and $\mathcal{O}(1)$. We will now show that two eigenvalues are complex conjugate with negative $\mathcal{O}(\delta)$ real part, and one is real, positive and $\mathcal{O}(1)$.

The characteristic polynomial is given by
$$
\lambda^3-\text{tr}(\mathcal{J}^*_2)\lambda^2+\frac{1}{2}[(\text{tr}(\mathcal{J}^*_2))^2-\text{tr}((\mathcal{J}^*_2)^2)]\lambda-\det(\mathcal{J}^*_2)=0
$$
where
$$
\text{tr}((\mathcal{J}^*_2)^2)=\left(\frac{\mathcal{R}_0^2(I^*_2)^2+\delta^2}{c^2}\right)+c^2>0
$$
so that the characteristic polynomial becomes
$$
P(\lambda)=\lambda^3-\left(c-\frac{\mathcal{R}_0I^*_2+\delta}{c}\right)\lambda^2+\left(\frac{\mathcal{R}_0I^*_2\delta}{c^2} -\mathcal{R}_0I^*_2-\delta\right)\lambda-\frac{(1+\delta)\mathcal{R}_0I^*_2}{c}=0.
$$
Applying the Routh-Hurwitz criterion to the characteristic polynomial, the corresponding Routh array (at leading order in $\delta$; recall that $I_2^*$ is $\delta$-small) is
\[ \begin{array}{c|cc} 
\lambda^3 & 1 & -(\mathcal{R}_0 I_2^\ast + \delta) + O(\delta^2) \\[2mm] 
\lambda^2 & - c + O(\delta) & -\frac{\mathcal{R}_0I^*_2}{c} + O(\delta^2) \\[2mm] 
\lambda^1 & - \frac{c^2 \mathcal{R}_0 I^*_2 + c^2 \delta + \mathcal{R}_0 I^*_2}{c^2} + O(\delta^2) & 0 \\[3mm] \lambda^0 & -\frac{\mathcal{R}_0I^*_2}{c} + O(\delta^2) & 0 \end{array} \]
Since $c, I_2^*>0$ and $\mathcal R_0>1$, all the entries in the first column except the leading one are negative. Hence the first column has signs $+,-,-,-$. Therefore there is exactly one sign change and, by the Routh-Hurwitz criterion, one root with positive real part. Furthermore, since $I_2^*=\mathcal O(\delta)$, all coefficients of $P(\lambda)$ except the coefficient of $\lambda^2$ 
are of order $\mathcal O(\delta)$. Consequently, the unstable eigenvalue $\lambda_1$ is $\mathcal O(1)$ and $\mathcal{O}(\delta)$-close to minus the coefficient of $\lambda^2$, while the remaining two eigenvalues lie in the left half-plane. Since their sum and product are both of order $\mathcal O(\delta)$, their real parts are negative and of order $\mathcal O(\delta)$. In order to show that these are complex conjugates, we reason as follows: since the remaining roots satisfy $|\lambda_{2,3}|\ll 1$, we can ignore the highest power $\lambda^3$ in the characteristic polynomial, and focus instead on

$$
P_2(\lambda)=-\left(c-\frac{\mathcal{R}_0I^*_2+\delta}{c}\right)\lambda^2+\left(\frac{\mathcal{R}_0I^*_2\delta}{c^2} -\mathcal{R}_0I^*_2-\delta\right)\lambda-\frac{1+\delta}{c} \mathcal{R}_0I^*_2=0.
$$
For this last quadratic polynomial, we may note
that the discriminant of the previous second degree polynomial in $\lambda$ is given by
$$
\Delta_2 = \left(\frac{\mathcal{R}_0I^*_2\delta}{c^2} -\mathcal{R}_0I^*_2-\delta\right)^2-4\frac{(1+\delta)\mathcal{R}_0I^*_2}{c}\left(c-\frac{\mathcal{R}_0I^*_2+\delta}{c}\right)<0
$$
since the first term is $\mathcal{O}(\delta^2)$ and the second is $\mathcal{O}(\delta)$. This allows us to conclude. 
\end{proof}
\begin{remark}
The presence of a complex-conjugate pair of eigenvalues at the EE accounts for the damped oscillations visible in the tail of the front (see Figure~\ref{fig:numerical_simulations}(a)): approaching the EE, the profile spirals rather than converging monotonically. 
\end{remark}

\subsubsection{Local dynamics for system \eqref{eq:PTW_SIRS_red_sp}}

Clearly, systems \eqref{eq:PTW_SIRS_red_sp} and \eqref{eq:PTW_SIRS_red_true} share the same equilibria. Moreover, since \eqref{eq:PTW_SIRS_red_true} is obtained from \eqref{eq:PTW_SIRS_red_sp} by reversing the traveling-wave coordinate, the spectra of the corresponding Jacobian matrices differ only by a change of sign. Therefore, under the assumptions of Lemmas~3.1 and~3.3, the Jacobian evaluated at the DFE possesses one real positive eigenvalue of order $\mathcal{O}(\delta)$ (corresponding with the $S$-axis) and two eigenvalues with negative $\mathcal{O}(1)$ real part. Similarly, the Jacobian evaluated at the EE has one real negative eigenvalue of order $\mathcal{O}(1)$ and a complex-conjugate pair whose real part is positive and of order $\mathcal{O}(\delta)$.

\subsection{Fast dynamics}
\label{sec:fast_lim}
Taking the limit $\delta\to 0$ in system \eqref{eq:PTW_SIRS_red_true}, we obtain the so-called \emph{layer problem}:
\begin{equation}
\begin{cases}
    \dfrac{{\rm d} S}{{\rm d} \xi}=-\dfrac{1}{c} \mathcal{R}_0 SI ,\medskip \\
    \dfrac{{\rm d} I}{{\rm d} \xi}=-J,\medskip \\
    \dfrac{{\rm d} J}{{\rm d} \xi}= cJ+\mathcal{R}_0 SI-I,
\end{cases}   
\label{eq:PTW_SIRS_layer}
\end{equation}
whose set of equilibria forms the \emph{critical manifold}, which will play a major role in our analysis:
\begin{equation}
\mathcal{C}_0:=\{ (S,I,J)\in \mathbb{R}_+^2\times \mathbb{R} \; | \; I=J=0 \}.
\label{eq:crit_manif}
\end{equation}
Note that, as long as $S,I>0$, the variable $S$ is strictly decreasing.

The Jacobian of \eqref{eq:PTW_SIRS_layer} is given by

$$
\mathcal{J}=\begin{pmatrix}
-\dfrac{\mathcal{R}_0 I}{c} & -\dfrac{\mathcal{R}_0 S}{c} & 0\\
0 & 0 & -1 \\
\mathcal{R}_0 I & \mathcal{R}_0 S-1 & c
\end{pmatrix}.
$$
Evaluated on $\mathcal{C}_0$, this becomes

$$
\mathcal{J}|_{\mathcal{C}_0}=\begin{pmatrix}
0 & -\dfrac{\mathcal{R}_0 S}{c} & 0\\
0 & 0 & -1 \\
0 & \mathcal{R}_0 S-1 & c
\end{pmatrix},
$$
which has eigenvalues $\lambda_0=0$ (corresponding to the $S$-axis, i.e. the critical manifold which, in this limit, is a line of equilibria) and
\begin{equation}
    \label{eq:lampm}
\lambda_{\pm} = \dfrac{c\pm\sqrt{c^2 - 4 \mathcal{R}_0 S + 4}}{2}.
\end{equation}
Since $c>0$, $\Re(\lambda_+)>0$ for any values of $S$ and $\mathcal{R}_0$. Note that the requirement that both eigenvalues $\lambda_{\pm}$ are real in the biologically feasible region $\Delta$ provides a constraint on the acceptable values of the traveling speed $c$, which is implied by the condition derived on the stability of the DFE in Section \ref{sec:lin_stab}. This will be relevant in Section \ref{sec:slow_flow}, as well. Instead, 
$$
\Re(\lambda_-) \lessgtr 0 \iff \mathcal{R}_0 S \lessgtr 1.
$$
As we shall see, near $\mathcal{C}_0$ (i.e., under the slow flow), $S$ increases and tends towards 1. Hence, orbits traveling close to the critical manifold move from its attracting part to its repelling part, and we might observe an entry-exit phenomenon. Moreover, note that for $S\in (0,1)$ we have $\lambda_+>\lambda_-$, so even though the system has two fast variables ($I$ and $J$), we fall within the standard setting for entry-exit functions \cite{de2016entry,kaklamanos2025entry,schecter2008exchange1,schecter2008exchange2}.

Note that system \eqref{eq:PTW_SIRS_layer} possesses a constant of motion (derived through a modification of a classical constant of motion of planar SIRS systems \cite{jardon2021geometric}, and which we can apply anywhere as long as $S>0$) given by:
\begin{equation}
    \label{eq:const_of_moti}
\Gamma(S,I,J)=-\dfrac{c}{\mathcal{R}_0} \log S + c (S+I)+J.
\end{equation}
This constant of motion provides a return relation from $\mathcal{C}_0$ to itself, explaining the excursion part (epidemic peaks) of all the ``loops'' around the EE; it characterizes the fast dynamics in a non-trivial way, as we explain in the following proposition.

\begin{proposition}\label{prop:fast_L}
Assume that $\mathcal{R}_0>1$. Then, there exists a unique
$L\in(0,1/\mathcal{R}_0)$ such that, for initial conditions close to $(S,0,0)$ with $S\in(1/\mathcal{R}_0,1]$, 
orbits of system \eqref{eq:PTW_SIRS_layer} tend towards the subset of $\mathcal{C}_0$ \eqref{eq:crit_manif}
 with $S\in [L,1/\mathcal{R}_0)$ as $\xi\to + \infty$. \end{proposition}
\begin{proof}
Recall that $S$ is decreasing in system \eqref{eq:PTW_SIRS_layer} as long as $S,I>0$. Orbits of said system will then traverse level curves of $\Gamma$ while decreasing the corresponding value of $S$. Consider $S_1 \in (1/\mathcal{R}_0,1]$; here, $S_1$ denotes the exit point from which the fast loop around the EE departs. The return point to $\mathcal{C}_0$ will be given by the non-trivial zero of 

$$
\hat{f}(x)= -\dfrac{c}{\mathcal{R}_0} \log x + c x - \left( -\dfrac{c}{\mathcal{R}_0} \log S_1 + c S_1  \right).
$$
For simplicity, we multiply the function by $1/c$, showing that the fast dynamics are completely dictated by the parameter $\mathcal{R}_0$:
\begin{equation}
 f(x)= -\dfrac{1}{\mathcal{R}_0} \log x + x - \left( -\dfrac{1}{\mathcal{R}_0} \log S_1 + S_1  \right).   
 \label{eq:entry_point}
\end{equation}
Clearly, $f(S_1)=0$. We cannot explicitly determine the second root of $f$, but we can reason as follows.
Note that
$$
f'(x)>0 \iff x>\dfrac{1}{\mathcal{R}_0},
$$
meaning that the non-trivial zero (the return point, which we denote by $S_\infty$) will satisfy $S_{\infty}< 1/\mathcal{R}_0$. See Figure \ref{fig:sketch} for a visualization of the concatenation of fast and slow dynamics; the latter will be discussed in the next section.

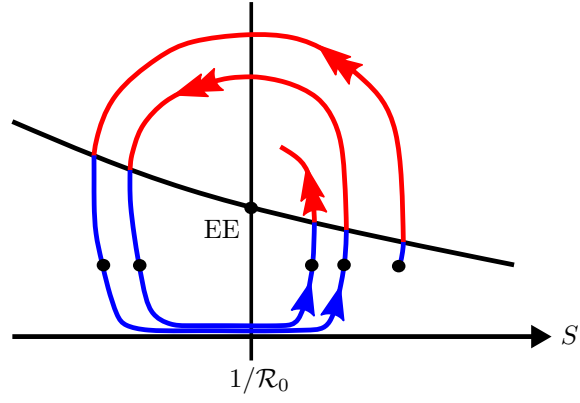
\begin{figure}[t!]
    \centering
\begin{tikzpicture}
 \node at (0,0) {\includegraphics[width=.45\textwidth]{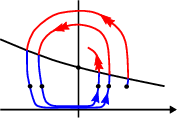}}; 
 \node at (3.8,-2) {$S$};
 \node at (-0.8,-0.6) {EE};
\node at (-0.3,-2.6) {$1/\mathcal{R}_0$}; 
\end{tikzpicture}
    \caption{Qualitative sketch of the dynamics, including fast loops and slow passages close to the critical manifold $\mathcal{C}_0$ (which will be studied in detail in Section \ref{sec:slow_flow}); the nullcline for $S$ is given by \eqref{eq:S_nullcline}. Compare with Figure \ref{fig:scales} for a quantitative representation of the orbits of system \eqref{eq:PTW_SIRS_red_true}. 
    \label{fig:sketch}}
\end{figure}

Since the orbit under study tends asymptotically to the DFE $(1,0,0)$ (as $\xi \to -\infty$), we want to find the point $(L,0,0)$ that connects to it. In order to do so, we need to solve 
$$
f(x)= -\dfrac{1}{\mathcal{R}_0} \log x + x-1=0.
$$
Since $\mathcal{R}_0>1$, the logarithm is ``squished'' towards the $x$-axis, bringing the intersection of its graph with the line passing through (1,0) (representing the DFE) closer to 0. Equivalently, we may look at
$$
- \log x + \mathcal{R}_0(x-1)=0,
$$
which shows that the line is secant to the graph of the logarithm (again, we are assuming $\mathcal{R}_0>1$). 

Lastly, since $S$ admits a positive limit as $\xi\to +\infty$, we show convergence to $\mathcal{C}_0$ as follows. Note that
$$S(0)-S(\xi)=\frac{\mathcal{R}_0}{c}\int_0^\xi S(s)I(s)\text{d}s \implies +\infty> S(0)-S_\infty = \frac{\mathcal{R}_0}{c}\int_0^{+\infty}S(s)I(s)\text{d}s, $$
which in turn implies $I\to 0$. By conservation of $\Gamma$ \eqref{eq:const_of_moti} and definition of $S_\infty$, $J\to 0$ as well.
\end{proof}

Proposition \ref{prop:fast_L} completely characterizes the dynamics of system \eqref{eq:PTW_SIRS_layer}; we will now investigate the behavior of the original perturbed system \eqref{eq:PTW_SIRS_red_true} close to its critical manifold $\mathcal{C}_0$ \eqref{eq:crit_manif}.

\subsubsection{Fast dynamics for system \eqref{eq:PTW_SIRS_red_sp}}

Clearly, taking $\delta\to 0$ in system \eqref{eq:PTW_SIRS_red_sp} would give us a system of ODEs that share with system \eqref{eq:PTW_SIRS_layer} both the critical manifold \eqref{eq:crit_manif} and the constant of motion \eqref{eq:const_of_moti}. Here, $S$ would be increasing, and the stability of the critical manifold would be inverted (attracting for $S\in (1/\mathcal{R}_0,1]$, repelling for $S\in [0,1/\mathcal{R}_0)$).

However, in this case, by continuity and monotonicity, the value $L\in (0,1/\mathcal{R}_0)$ provides a lower bound for the acceptable initial conditions of the fast dynamics; for $S_1<L$, the orbits would exit the biologically feasible region $\Delta$ \eqref{eq:biol_feas} before returning (asymptotically) to $\mathcal{C}_0$ \eqref{eq:crit_manif}. Hence, not all initial conditions close to the intersection of the biologically relevant region and the critical manifold would result in orbits which remain in the biologically relevant region as the system evolves.

\subsection{Slow flow and entry-exit mechanism}
\label{sec:slow_flow}

We now want to analyze how the dynamics evolve close to the critical manifold $\mathcal{C}_0$ \eqref{eq:crit_manif}.

In order to do so, we ``zoom'' close to $\mathcal{C}_0$ by introducing the rescaling $\delta v:=I$, $\delta w:=J$; this rescaling, which is only meaningful in an $\mathcal{O}(\delta^n)$-neighborhood of $\mathcal{C}_0$, with $n\geq 1$, brings system \eqref{eq:PTW_SIRS_red_true} to
\begin{equation}
\begin{cases}
    \dfrac{{\rm d} S}{{\rm d} \xi}=\dfrac{\delta}{c} \left((1-S-\delta v) -\mathcal{R}_0 Sv \right) - \dfrac{\delta^2}{c^2} w,\medskip\\
    \dfrac{{\rm d} v}{{\rm d} \xi}=-w,\medskip\\
    \dfrac{{\rm d} w}{{\rm d} \xi}= cw+\mathcal{R}_0 Sv-v.
\end{cases}   
\label{eq:PTW_SIRS_resc_loc}
\end{equation}
Note that, in these coordinates, $\mathcal{C}_0$ is given by $v=w=0$. After introducing the slow variable $\zeta:=\delta \xi$, the dynamics on $\mathcal{C}_0$ are described by the equation
\begin{equation}\label{eq:lenta_S}
    \dfrac{{\rm d} S}{{\rm d} \zeta}=\dfrac{1}{c} (1-S),
\end{equation}
which can be solved explicitly (recall that the initial value is approximated up to $\mathcal{O}(\delta)$ by $S_\infty$ defined in the previous section)
\begin{equation}\label{eq:sol_lenta_S}
S(\zeta)=1+(S_\infty-1) e^{-\zeta/c}.
\end{equation}
This implies that, as $\zeta \to +\infty$, the susceptible population $S\to 1$. 

We remark that system \eqref{eq:PTW_SIRS_resc_loc} is in standard GSPT form; in particular, it is in a form to which we can directly apply the so-called \emph{entry-exit function} \cite{de2016entry,kaklamanos2025entry,schecter2008exchange1,schecter2008exchange2}. Let us briefly recall, in informal terms, the mechanism captured by this function; for a more detailed explanation, we refer the interested reader to the references listed a few lines above. Keep in mind Figure \ref{fig:entry-exit} for a qualitative representation of this Poincaré map, as a visual support for the forthcoming description. Consider a trajectory that, following a fast excursion described in Section \ref{sec:fast_lim}, lands in a small neighborhood of an \emph{attracting} portion of the critical manifold, i.e.\ a portion along which the relevant fast eigenvalue is negative. Under the slow flow, the trajectory drifts along the manifold and is exponentially contracted towards it; however, as it crosses a non-hyperbolic point (here, the point $S=1/\mathcal{R}_0$, where the eigenvalue $\lambda_-$ \eqref{eq:lampm} changes sign) the manifold turns from attracting into repelling. The trajectory does not leave immediately: because of the contraction accumulated on the stable side, it remains close to the (now unstable) manifold for a further $\mathcal{O}(1)$ amount of slow time, until the accumulated expansion exactly balances the previous contraction. Only then does it depart, triggering a new fast excursion. The entry-exit function is precisely the relation that maps the entry point of this passage to the corresponding exit point, or equivalently determines the exit \emph{time} at which the trajectory leaves the neighborhood of the manifold. This phenomenon, known in the literature also as \emph{delayed loss of stability} or \emph{bifurcation delay}, is quantified by balancing the signed accumulated rate $\lambda_-$ along the slow trajectory, which is exactly the content of the integral condition \eqref{eq:exit_time} below. In our epidemic setting, it has a transparent interpretation: it measures the time (and hence the spatial extent, in the traveling front setting of this work) over which the infected population stays negligibly small between two consecutive outbreaks. 

\begin{figure}[t!]
\centering
  \begin{tikzpicture}
 \node at (0,0) {\includegraphics[width=.45\linewidth]{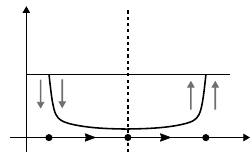}};
 \node at (2.8,0.3) {$r=\Tilde{r}$}; 
 \node at (3.5,-2.1) {$S$};
 \node at (-3.2,2) {$v,w$};
  \node at (2.4,-2.1) {$S_E$};
  \node at (-2.15,-2.1) {$S_\infty$}; 
\node at (0.2,-2.5) {$1/\mathcal{R}_0$}; 
\end{tikzpicture}
\caption{Visualization of the entry–exit map on the set $\{r=\Tilde{r}\}$, which indicates a (small) radial distance from the $S$-axis, illustrating delayed loss of stability. The point $S=1/\mathcal{R}_0$ is the unique non-hyperbolic point on the critical manifold. The point $S=S_E$ corresponds to the position of $S(\zeta)$ as it evolved from $\zeta=0$ to $\zeta=T_E$, with $T_E$ being the exit time discussed in this section. Refer to Section \ref{sec:num_valid} for numerical implementation of this function, and comparison with the simulations provided in Section \ref{sec:prelim}. \label{fig:entry-exit}}\end{figure}
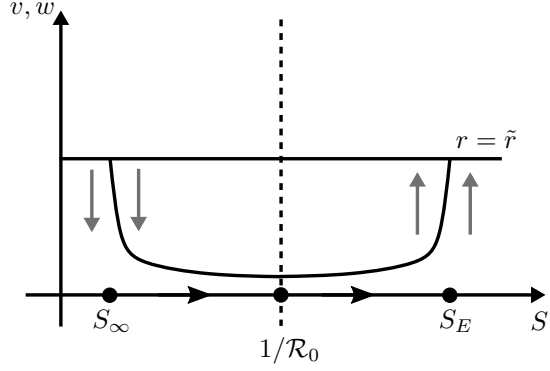

To be more precise, we are able to estimate the permanence of the dynamics in a cylindrical neighborhood of $\mathcal{C}_0$, rather than a coordinate relation between an entry point in the slow flow (approximated by $S_\infty$) and the corresponding exit point. Moreover, the eigendirection corresponding to $\lambda_-$ does not coincide with either $I$ or $J$.

We can obtain the exit \emph{time} by integrating 
$\lambda_-$ along solutions of the slow flow for $S$  \eqref{eq:sol_lenta_S} and finding the non-trivial zero of the corresponding integral equation
\begin{equation}
\int_{0}^{T_E} \left(c-\sqrt{c^2 - 4 \mathcal{R}_0 S(\zeta) + 4}\right) \text{d}\zeta=0.
\label{eq:exit_time}
\end{equation}

By introducing, for ease of notation, $a:= c^{2}-4(\mathcal{R}_0-1)\geq 0$  (since $c\geq c_{\text{crit}}$; recall \eqref{eq:vel_crit}) and $B:= 4\mathcal{R}_0(S_\infty-1)<0$, we can re-write the argument of the square root as $\rho(S(\zeta))=a-B\,e^{-\zeta/c}$. Then, the integral equation for the exit time can be rewritten as the non-trivial zero of 
\begin{equation}\label{eq:G1}
  G(T):=\int_0^{T}\left(c-\sqrt{\rho\left(S(\zeta)\right)}\right)\text{d}\zeta
        =\int_0^{T}\left(c-\sqrt{a-B\,e^{-\zeta/c}}\right)\text{d}\zeta .
\end{equation}
From Section \ref{sec:fast_lim} we know that the integrand of \eqref{eq:G1} is negative for $S<1/\mathcal{R}_0$ and positive for $S>1/\mathcal{R}_0$; hence, $G(0)=0$, $G$ decreases until
$T=T_c$ (where $S(T_c)=1/\mathcal{R}_0$) and then increases. Since $S(\zeta)\to 1$ as $\zeta\to+\infty$, we have that in the same limit $G(T)\to+\infty$. In particular, a zero $T_E$ of $G(T)$ satisfying
$T_E>T_c$ always exists. The exit point will satisfy $S_E \in (1/\mathcal{R}_0, 1)$. The unique positive root of \eqref{eq:exit_time} can be found numerically and compared to the simulations of system \eqref{eq:PTW_SIRS_red_true}, as we will do in Section \ref{sec:num_valid}. 

Note that, by continuity and monotonicity, orbits starting from \emph{larger} values of $S_\infty \in (0, 1/\mathcal{R}_0)$ will necessarily exit from \emph{smaller} values of $S_E \in (1/\mathcal{R}_0,1)$. Combining this with the fact that smaller values of $S_E$ correspond to smaller values of the constant of motion $\Gamma$, we see that subsequent combination of fast and slow dynamics will bring the orbit closer and closer to the EE. We provide a more rigorous explanation in the following section.

\subsection{Singular skeleton and candidate orbit}
\label{sec:singular_full}

The results obtained in Sections \ref{sec:lin_stab}-\ref{sec:slow_flow} provide a complete geometric description of the singular dynamics underlying the traveling front emerging in system \eqref{eq:PTW_SIRS_red_true}. In particular, the local analysis near the equilibria identifies the asymptotic behavior at the endpoints of the front, the layer problem determines the fast outbreak excursions through the invariant $\Gamma$, and the entry-exit mechanism describes the slow drift along the critical manifold. Combining these ingredients yields a singular orbit composed of alternating fast and slow segments entering a neighborhood of the EE where the local dynamics take over (see Figure \ref{fig:qualitative}). This singular geometric trajectory will serve as the skeleton of the traveling front and is summarized in the following proposition.

\begin{proposition}[Singular front structure] \label{prop:singular_front} 
Let $\mathcal{R}_0>1$ and $c>0$. In the singular limit $\delta\to 0$, there exists at least one migration speed $c\geq c_{\text{crit}}$ \eqref{eq:vel_crit} that gives rise to a singular front orbit $\Phi_0$ connecting the DFE to the EE, corresponding to a singular solution of \eqref{eq:SIRS_red_sp_adim}. The singular front $\Phi_0$ is obtained by concatenating orbit segments of the reduced and layer problems associated with \eqref{eq:PTW_SIRS_red_true}. In particular, \[ \Phi_0 = \bigcup_{i=1}^{n} \left( \phi_{\rm fast}^{(i)} \cup \phi_{\rm slow}^{(i)} \right) \cup \phi_{\rm loc}, \] where $n\in\mathbb N$ denotes the number of fast-slow cycles performed by the orbit before it enters the local dynamical regime in a neighborhood of the EE and
\begin{enumerate} 
\item[(i)] $\phi_{\rm fast}^{(i)}$ is a fast excursion of the layer problem \eqref{eq:PTW_SIRS_layer}, leaving $\mathcal{C}_0$ (for $i=1$, from a small neighborhood of the DFE) along a level set of the invariant $\Gamma$, and re-intersecting  $\mathcal{C}_0$ at the point approximated by $S_\infty<1/\mathcal{R}_0$ determined by \eqref{eq:entry_point}; 
\item[(ii)] $\phi_{\rm slow}^{(i)}$ is a slow orbit segment lying both on the attracting and on the repelling branch of $\mathcal{C}_0$, governed by the reduced flow \eqref{eq:lenta_S} and evolving for a time approximated by \eqref{eq:exit_time}; 
\item[(iii)] $\phi_{\rm loc}$ is a local orbit segment in a neighborhood of the EE, organized by the complex-conjugate eigenvalues described in Lemma~\ref{lem:local_stab_EE}. \end{enumerate} The singular orbit $\Phi_0$ forms a front connecting the EE (behind) to the DFE (ahead), and the number of fast excursions is controlled by the entry-exit mechanism of Section~\ref{sec:slow_flow}. \end{proposition}

\begin{figure}[h!]
    \centering
    \includegraphics[width=1\linewidth]{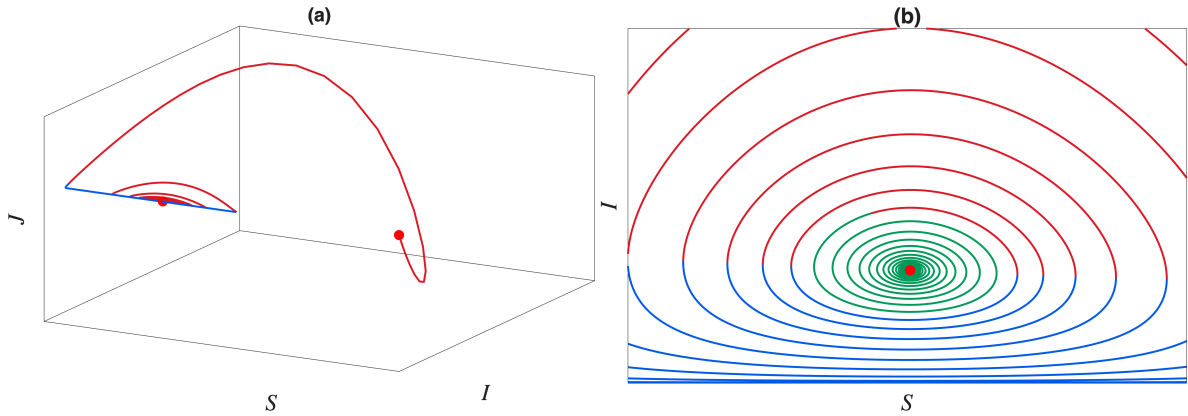}
    \caption{Schematic representation of the geometric structure of the traveling front made by alternating slow and fast phases. \textbf{(a)} Orbit in the full three dimensional $(S,I,J)$ space; \textbf{(b)} zoom-in close to the EE in the $(S,I)$ plane. Different colors identify the different dynamical regimes: local dynamics (green) near the EE, slow dynamics (blue) close to the critical manifold, and fast outbreak excursions (red) governed by the layer problem.
    \label{fig:qualitative}}
\end{figure}

\begin{remark}
The analysis carried out thus far allows us to define a proper forward invariant set, which provides a more precise region in which the dynamics of system \eqref{eq:PTW_SIRS_red_true} evolves, compared to the biologically relevant region $\Delta$ \eqref{eq:biol_feas}.
Recall Proposition \ref{prop:fast_L} and \eqref{eq:const_of_moti}; we can define the forward invariant set, for some $K>0$,
$$\Omega :=\left\{ (S,I,J) \in \mathbb{R}^3 \; | \; \Gamma(S,I,J)\leq \Gamma(1,0,0)+K\delta,\; L\leq S\leq 1,\;0\leq I\leq 1,\; S+I+\frac{J}{c} \leq 1 \right\}.$$
To see that the first inequality holds, observe that
$$
\dfrac{\text{d}\Gamma}{\text{d}\xi}= \delta\left( 1-\frac{1}{\mathcal{R}_0S}  \right)R,
$$
which increases and then decreases by an $\mathcal{O}(\delta)$ value during the first fast excursion.
\end{remark}

\begin{remark}[Persistence for small $\delta>0$] The construction above is carried out in the singular limit $\delta=0$ and therefore provides a singular front $\Phi_0$ obtained by concatenating fast, slow and local orbit segments. A rigorous persistence result for the corresponding traveling front when $\delta>0$ is beyond the scope of the present work. Indeed, the critical manifold $\mathcal C_0$ is a line of equilibria and loses normal hyperbolicity at $S=1/\mathcal R_0$, so that the standard Fenichel theory does not directly apply to the entire orbit. Nevertheless, the fast segments persist as $\mathcal O(\delta)$ perturbations of the layer dynamics, while the passage near the non-hyperbolic point is quantitatively described by the entry-exit framework developed e.g. in \cite{de2008smoothness,de2016entry}.
\end{remark}

\subsubsection{Slow flow and entry-exit mechanism for system \eqref{eq:PTW_SIRS_red_sp}}

This short subsection is the main reason for our back-and-forth between the detailed analysis of system \eqref{eq:PTW_SIRS_red_true} and system \eqref{eq:PTW_SIRS_red_sp}. Indeed, as we shall show below, orbits of system \eqref{eq:PTW_SIRS_red_sp} might, in principle, exit the biologically relevant region during the slow flow by crossing into the region representing a \emph{negative} susceptible population.

Introducing the same change of coordinates $\delta v:=I$, $\delta w:=J$ in system \eqref{eq:PTW_SIRS_red_sp}, the slow dynamics on $\mathcal{C}_0$ are instead given by
$$
    \dfrac{{\rm d} S}{{\rm d} \zeta}=\dfrac{1}{c} (S-1),
$$
so that 
\begin{equation}\label{eq:S_zeta}
S(\zeta)=1+(S_\infty-1) e^{\zeta/c}.
\end{equation}
This implies that $S\to -\infty$ as $\zeta \to +\infty$. We define the exact values of $\zeta$ for which $S$ crosses into the non-biologically feasible region, namely
\begin{equation}\label{eq:zeta_star}
S(\zeta^*)=0 \iff \zeta^*=-c \log (1-S_\infty).
\end{equation}
Note that the closer $S_\infty$ is to 1, the longer the corresponding orbit remains in the biologically feasible region.

Following a similar construction as in the previous section, we obtain the following entry-exit relation:
$$
\int_{0}^{T_E} \left(-c+\sqrt{c^2 - 4 \mathcal{R}_0 S(\zeta) + 4}\right) \text{d}\zeta=0.
$$
One would naturally hope that, given any $S_\infty \in (1/\mathcal{R}_0,1)$, the corresponding exit time satisfies $T_E<\zeta^*$ \eqref{eq:zeta_star}; however, as we shall show below, this is not always the case.

By introducing, for ease of notation, $a:= c^{2}-4(\mathcal{R}_0-1)\geq 0$ (since $c\geq c_{\text{crit}}$; recall \eqref{eq:vel_crit}) and $B:= 4\mathcal{R}_0(S_\infty-1)<0 $, we can re-write the argument of the square root as $\rho(S(\zeta))=a-B\,e^{\zeta/c}$. Then, the integral equation for the exit time can be rewritten as the non-trivial zero of 
\begin{equation}\label{eq:G}
  G(T):=\int_0^{T}\left(-c+\sqrt{\rho\left(S(\zeta)\right)}\right)\text{d}\zeta
        =\int_0^{T}\left(-c+\sqrt{a-B\,e^{\zeta/c}}\right)\text{d}\zeta .
\end{equation}
From Section \ref{sec:fast_lim} we know that the integrand of \eqref{eq:G} is negative for $S>1/\mathcal{R}_0$ and positive for $S<1/\mathcal{R}_0$; hence, $G(0)=0$, $G$ decreases until
$T=T_c$ (where $S(T_c)=1/\mathcal{R}_0$) and then increases. Since $S(\zeta)\to -\infty$ as $\zeta\to+\infty$, we have that in the same limit $G(T)\to+\infty$. In particular, a zero $T_E$ of $G(T)$ satisfying
$T_E>T_c$ always exists. 

Introducing, for ease of notation, $\omega :=\sqrt{a-B\,e^{\zeta/c}}$, we can find a primitive of \eqref{eq:G} as 
\begin{equation}\label{eq:prim}
  \int\sqrt{a-B\,e^{\zeta/c}}\,\text{d}\zeta
  = c\!\left[\,2\omega+\sqrt a\,\ln\frac{\omega-\sqrt a}{\omega+\sqrt a}\right]
  .
\end{equation}
Evaluating this expression between $0$ and $T$, we obtain
\begin{equation}\label{eq:GT}
  G(T)=-cT+2c\bigl(\omega(T)-\omega(0)\bigr)
       +c\sqrt a\left[\ln\frac{\omega(T)-\sqrt a}{\omega(T)+\sqrt a} -\ln\frac{\omega(0)-\sqrt a}{\omega(0)+\sqrt a}\right].
\end{equation} 
The unique positive root of \eqref{eq:GT} can be found numerically and compared to the simulations of system \eqref{eq:PTW_SIRS_red_sp}. The expression is hardly treatable analytically, except for the following considerations.

One condition to ensure that the exit time corresponds to a positive exit value of $S$ is
$$
G(\zeta^*)>0
  \;\Longrightarrow\;
  T_E<\zeta^*
  \;\Longrightarrow\;
  S(T_E)>S(\zeta^*)=0 .$$
However, as we shall see in the following proposition, this is not always the case for any $S_\infty \in (1/\mathcal{R}_0,1)$. 

\begin{proposition}\label{prop:S_inftystar}
Assume that $\mathcal{R}_0>1$, $c>0$ and $a=c^{2}-4(\mathcal{R}_0-1)>0$. Then, there exists a unique
$S_\infty^*\in(1/\mathcal{R}_0,1)$ such that
\[
  G(\zeta^*)>0 \iff 0<S_\infty<S_\infty^*,
  \qquad
  G(\zeta^*)<0 \iff S_\infty^*<S_\infty<1 .
\]
The inequality $G(\zeta^*)>0$ does not hold on the whole interval
$1/\mathcal{R}_0<S_\infty<1$, but only in the sub-interval $1/\mathcal{R}_0<S_\infty<S_\infty^*$.
\end{proposition}
\begin{proof}
It is convenient to rewrite $G(\zeta^*)$ with the change of variable
$\zeta\mapsto S$. From \eqref{eq:S_zeta} and/or \eqref{eq:zeta_star} we have
$\text{d}\zeta=\dfrac{c}{S-1}\text{d}S$ and the integration interval
$\zeta\in[0,\zeta^*]$ corresponds to $S\in[0,S_\infty]$, from which we can write
\begin{equation}\label{eq:GPhi}
  \frac{G(\zeta^*)}{c}
  =\Phi(S_\infty):=\int_0^{S_\infty} m(S)\,\text{d} S,
  \quad \hbox{with} \quad
  m(S):=\frac{\sqrt{c^{2}+4(1-\mathcal{R}_0 S)}-c}{1-S}.
\end{equation}
On $[0,1)$ we have $1-S>0$, hence the sign of $m(S)$ corresponds to the sign of
$\sqrt{c^2+4(1-\mathcal{R}_0 S)}-c$, which is the same as $(1-\mathcal{R}_0 S)$. Hence,

$$
  m(S)>0 \text{ for } S \in (0,1/\mathcal{R}_0),
  \qquad
  m(S)<0 \text{ for } S \in (1/\mathcal{R}_0,1).
$$
This implies that $\Phi$ is strictly increasing on $[0,1/\mathcal{R}_0]$ and strictly decreasing on $[1/\mathcal{R}_0,1)$, with $\Phi(0)=0$ and
$\Phi(1/\mathcal{R}_0)>0$. Lastly, in the limit $S\to1^-$, we observe

$$
  m(S)\sim\frac{\sqrt a-c}{1-S},
$$
and $\sqrt a<c$ implies $\Phi(S_\infty)\to-\infty$ as $S_\infty\to1^-$. In summary, $\Phi(S)$ is a continuous function satisfying $\Phi(1/\mathcal{R}_0)>0$; for values of $S>1/\mathcal{R}_0$, it decreases and tends to $-\infty$ as $S\to 1^-$. This implies that $\Phi(S)$ has a unique root $S_\infty^*\in(1/\mathcal{R}_0,1)$. We conclude by recalling that $G(\zeta^*)=c \,\Phi(S_\infty)$.
\end{proof}

\subsection{Numerical validation}
\label{sec:num_valid}

In this section, we compare the analytical results obtained in Sections~\ref{sec:lin_stab}-\ref{sec:slow_flow} with the class of numerical simulations of system \eqref{eq:SIRS_red_sp_adim} introduced in Section~\ref{sec:prelim}. Our goal is to provide a geometric and quantitative validation of the slow-fast structure underlying the observed traveling front solutions depicted in Figure~\ref{fig:numerical_simulations}.

\begin{figure}[b!]
    \centering
    \includegraphics[width=1\linewidth]{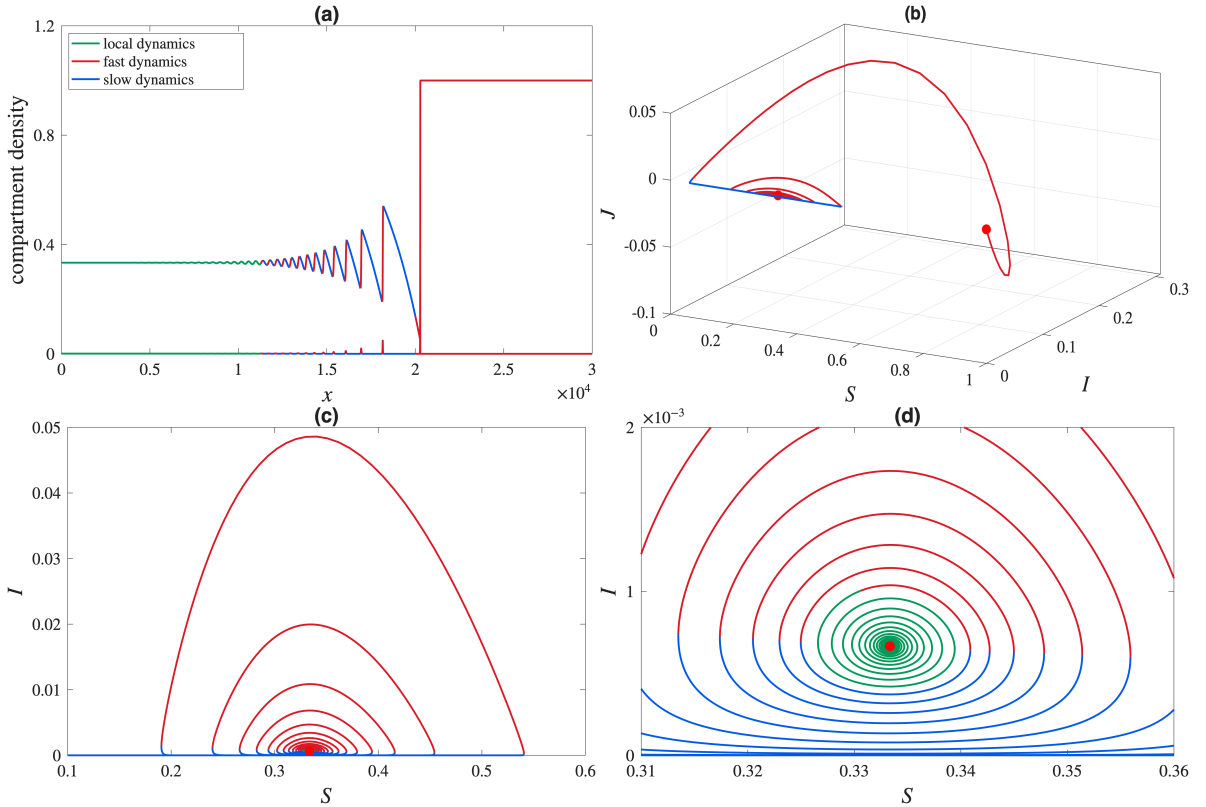}
    \caption{Multiscale structure of the traveling front solution obtained from numerical simulations of system \eqref{eq:SIRS_red_sp_adim}. The colors distinguish the different dynamical regimes, with local (green curves), fast (red), slow (blue)  dynamics corresponding to distinct portions of the orbit. {\bf(a)} Spatial profiles at $t = 6000$, highlighting the coexistence of different spatial scales: a local regime near the equilibrium, a slow evolution region, and a sharp fast transition layer. {\bf(b)} Corresponding trajectory in the phase space $(S,I,J)$, illustrating the concatenation of slow motion near the critical manifold with fast excursions governed by the layer dynamics. {\bf(c)} Projection of the orbit in the $(S,I)$ plane, showing the transition between slow and fast dynamics. {\bf(d)} Zoom of the orbit near the equilibrium, emphasizing the transition from the local regime to the fast dynamics. Parameter values as in Figure \ref{fig:numerical_simulations}.}
    \label{fig:scales}
\end{figure}

Let us begin by illustrating the global organization of the dynamics through Figure \ref{fig:scales}. This figure provides a direct numerical visualization of the multiscale structure predicted by the analysis in Section \ref{sec:analysis}, where the evolution of the solution was decomposed into local, fast, and slow regimes. 

Panel \ref{fig:scales}\textbf{(a)} shows the spatial profile of the traveling front at a fixed final time\linebreak ($t=6000$). Three distinct regions can be clearly identified: a local regime in a neighborhood of the EE (green curves), where the dynamics are governed by the linearization described in Section \ref{sec:lin_stab}; a sharp fast transition layer (red curves), corresponding to the excursion governed by the layer problem analyzed in Section \ref{sec:fast_lim}; and a slow evolution region (blue curves), in which the solution remains close to the critical manifold $\mathcal{C}_0$ and evolves according to the reduced dynamics derived in Section \ref{sec:slow_flow}. From a biological perspective, this decomposition reflects the coexistence of different mechanisms driving the spatial spread of the infection. Moving from left to right along the front, one first observes a regime in a neighborhood of the EE, where both susceptible and infected populations are at positive levels and the infection is well established. This regime is followed by a sequence of slow segments and fast excursions that govern the invasion process. In detail, along the slow segments the dynamics correspond to a gradual adaptation of the susceptible population as the infection starts to build up. These regions are interspersed with fast transitions, which represent localized outbreak events: the number of infected individuals rapidly increases, reaches a peak, and subsequently decreases, while the susceptible population simultaneously recovers. This repeated interplay between slow adjustment and fast outbreak dynamics provides a coherent biological interpretation of the multiscale structure of the traveling front. As one moves further along the spatial domain, the solution reaches a region in which the infection has completely vanished and the population consists entirely of susceptible individuals. It should be noted that the front propagates with positive migration speed, hence traveling to the right, and can therefore be interpreted as an invasion front in which the EE persists behind the front, while the DFE is ahead of it.

Then, panel \ref{fig:scales}\textbf{(b)} displays the corresponding trajectory in the full phase space $(S,I,J)$. The orbit exhibits a clear concatenation of slow and fast segments, in agreement with the geometric decomposition introduced in Section \ref{sec:analysis}. In particular, the solution evolves close to the critical manifold $\mathcal{C}_0=\{J=I=0\}$ for a relatively long portion of the dynamics, before undergoing a fast excursion that brings it away from the manifold and subsequently back towards it. 

This behavior is more clearly visible in panel \ref{fig:scales}\textbf{(c)}, where the orbit is represented in a zoomed-in portion of the $(S,I)$ plane. The transition between slow and fast dynamics appears as a sharp change in direction of the trajectory, providing numerical evidence of the entry-exit scenario described in Section~\ref{sec:slow_flow}. 

Finally, panel \ref{fig:scales}\textbf{(d)} further zooms into the neighborhood of the equilibrium, showing how the local dynamics smoothly connect with the fast layer, thus completing the geometric structure of the orbit. We remark that panels \ref{fig:scales}\textbf{(c)} and \ref{fig:scales}\textbf{(d)} closely resemble the qualitative structure of the orbit depicted in the sketch of Figure~\ref{fig:sketch}, thereby providing a direct numerical counterpart to the geometric scenario derived in Section~\ref{sec:analysis}. 

Overall, Figure~\ref{fig:scales} confirms that the traveling front solution is organized as a concatenation of distinct dynamical regimes, in full agreement with the slow-fast framework developed in Section~\ref{sec:analysis}.

\begin{figure}[b!]
    \centering
    \includegraphics[width=1\linewidth]{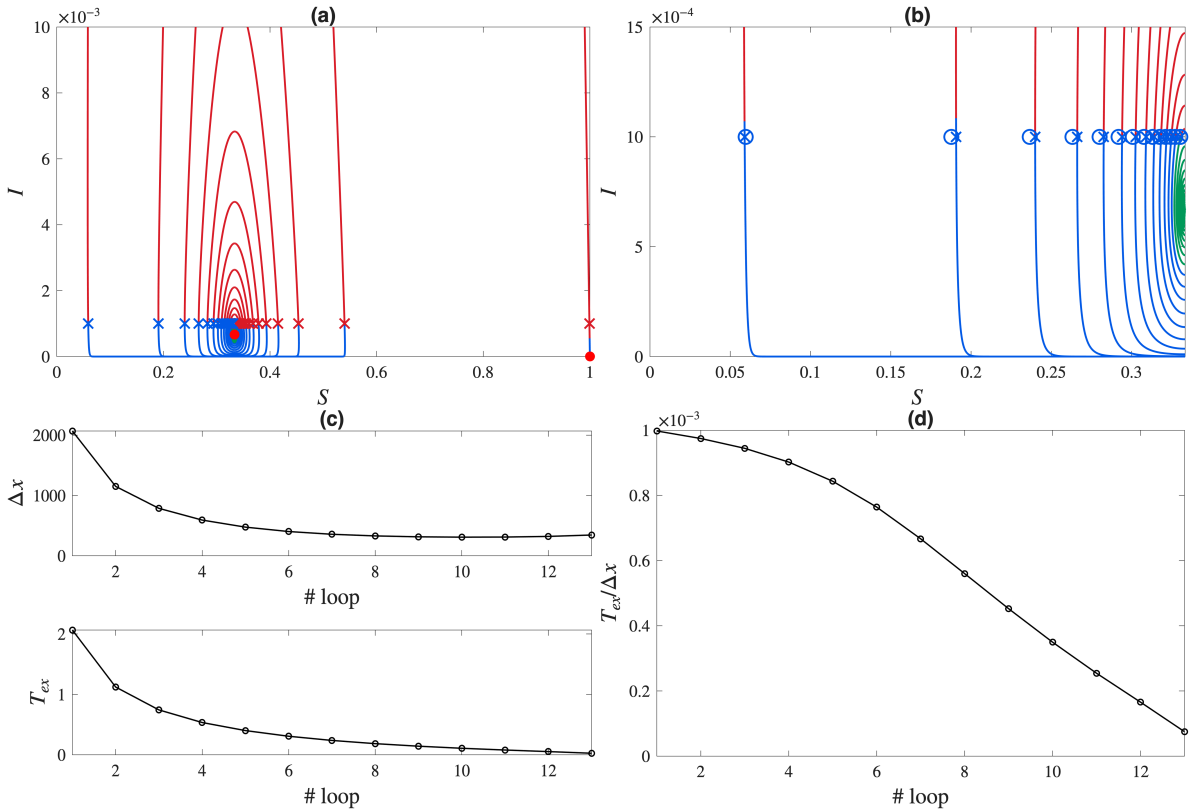}
    \caption{Numerical validation of the entry-exit mechanism underlying the slow-fast dynamics of system \eqref{eq:SIRS_red_sp_adim}. {\bf(a)} Projection of the orbit in the $(S,I)$ plane, highlighting the transition between slow and fast dynamics. Blue crosses indicate entry points into the slow flow (transition from fast to slow dynamics), while red crosses denote numerically detected exit points (transition from slow to fast dynamics). {\bf(b)} Zoom of the entry dynamics, where blue circles mark the numerical prediction of the entry point location predicted by \eqref{eq:entry_point}. {\bf(c)} Comparison between the exit time $T_E$ obtained from the integral condition \eqref{eq:exit_time} and the spatial extent $\Delta x$ of the corresponding excursion along the slow critical manifold, providing a quantitative link between the analytical prediction and the numerical trajectory. {\bf(d)} Ratio between $T_E$ and $\Delta x$, exhibiting a scaling law proportional to $O(\delta)$. Parameter values as in Figure \ref{fig:numerical_simulations}.} 
    \label{fig:entryexit_validation}
\end{figure}

Let us now move to a more quantitative validation of the (fast) constant of motion and (slow) entry-exit mechanism described in Sections \ref{sec:fast_lim}-\ref{sec:slow_flow}, which governs the transition between slow and fast dynamics. This mechanism plays a central role in determining how trajectories approach and leave a neighborhood of the critical manifold, and ultimately controls the multiscale structure of the traveling front. This behavior is illustrated in Figure \ref{fig:entryexit_validation}. 

In particular, panel \ref{fig:entryexit_validation}\textbf{(a)} shows the trajectory in the $(S,I)$ plane, where the transition between slow and fast regimes is clearly visible. Red crosses denote the numerically detected exit points from the slow dynamics, corresponding to transitions from slow to fast motion, while blue crosses identify the entry points, marking the transition from fast motion back to slow permanence close to $\mathcal{C}_0$. More precisely, we exploit the fact that the slow dynamics take place in an $\mathcal{O}(\delta)$ neighborhood of the critical manifold $\mathcal{C}_0$. In order to numerically separate the slow and fast regimes, we consider the dynamics in the full phase space $(S,I,J)$ and introduce a cylindrical neighborhood of radius $\delta$ centered around $\mathcal{C}_0$. Therefore, numerical entry and exit points are identified as the intersections of the trajectories with this cylinder. Consequently, the red and blue markers displayed in panel \ref{fig:entryexit_validation}\textbf{(a)} correspond to the projections of these intersection points onto the $(S,I)$ plane. The observed structure is in  excellent agreement with the geometric picture provided by the analysis of the layer problem and the reduced flow shown in Figure \ref{fig:sketch}. 

Then, a more detailed view of the entry dynamics is provided in panel \ref{fig:entryexit_validation}\textbf{(b)}, where we display a zoom of the trajectory near the slow manifold. The blue circles represent the theoretical prediction of the entry location, obtained from the relation derived in \eqref{eq:entry_point}. The agreement is quite good for all the excursions, consistent with the fact that our analytical description approximates the transition between fast and slow motion. Indeed, in the numerical procedure we identify entry and exit points as intersections with a cylinder of radius exactly $\delta$ around the critical manifold $\mathcal{C}_0$, whereas the theoretical analysis only predicts that such transitions occur at a distance of order $\mathcal{O}(\delta)$ from $\mathcal{C}_0$. As a consequence, one expects the presence of a transitional layer within this neighborhood $\mathcal{O}(\delta)$-close to $\mathcal{C}_0$ in which the dynamics gradually switch from fast to slow motion, and vice versa. 

Moreover, panel \ref{fig:entryexit_validation}\textbf{(c)} shows a comparison between the exit time $T_E$, computed as the solution of the integral condition \eqref{eq:exit_time}, and the spatial extent $\Delta x$ of the corresponding permanence of the orbit close to $\mathcal{C}_0$. These quantities provide a quantitative measure of how closely the trajectory follows the critical manifold $\mathcal{C}_0$, comparing the analytical prediction with the numerically observed dynamics along the slow excursion. 

Finally, panel \ref{fig:entryexit_validation}\textbf{(d)} displays the ratio between $T_E$ and $\Delta x$. This ratio is closer to $\delta$ the longer the corresponding slow part of the orbit is - that is, the clearer the separation between slow flow and local dynamics around the EE is. As the orbit performs successive excursions around the EE, thereby entering a regime in which the slow-fast structure  slowly loses influence in favor of the local dynamics, this ratio progressively becomes smaller until the local dynamics completely dictate the behavior of the system and the orbit does not enter a small enough neighborhood of the critical manifold anymore.
From a biological perspective, this analysis indicates that the duration of the slow adaptation phase and the spatial extension of the corresponding region are directly controlled by the rate of loss of immunity. Specifically, as $\delta$ decreases, the system remains longer in a regime of gradual adjustment close to the slow manifold, resulting in wider regions separating successive outbreak events.

\begin{figure}[t!]
    \centering
    \includegraphics[width=1\linewidth]{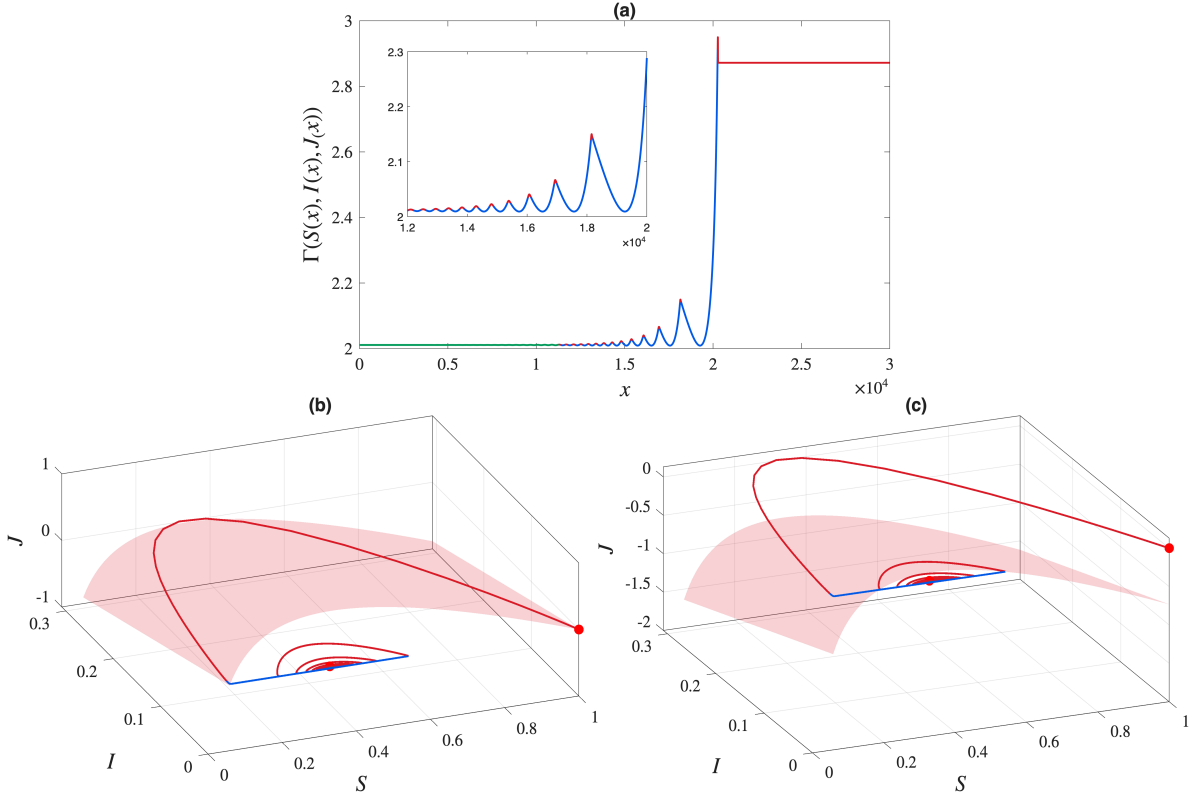}
    \caption{Numerical validation of the invariant quantity $\Gamma$ introduced in Section \ref{sec:fast_lim}. \textbf{(a)} Value of $\Gamma(S(\xi),I(\xi),J(\xi))$ along the traveling front profile. \textbf{(b)}, \textbf{(c)} Orbit in the $(S,I,J)$ phase space together with representative level sets of $\Gamma$ corresponding to two distinct exit points. Parameter values as in Figure \ref{fig:numerical_simulations}.}
    \label{fig:fast_first_integral}
\end{figure}

Let us now turn our attention to the geometric structure of the fast dynamics, as characterized by the invariant quantity $\Gamma$ introduced in Section \ref{sec:fast_lim}. Recall that $\Gamma$ is an exact first integral of the layer problem and therefore provides a natural geometric framework for describing the fast excursions away from the critical manifold. The numerical validation of this mechanism is reported in Figure \ref{fig:fast_first_integral}. 

In detail, panel \ref{fig:fast_first_integral}\textbf{(a)} shows the quantity $\Gamma(S(\xi),I(\xi),J(\xi))$ evaluated along the numerical traveling front. As expected, $\Gamma$ remains nearly constant during the fast excursions (red arcs), while it varies during the slow phases close to the critical manifold (blue arcs). This behavior appears as a sequence of almost flat plateaus separated by gradual transitions. The inset, which highlights the outbreak region, further emphasizes that each rapid epidemic excursion takes place with only minimal variation of $\Gamma$. This provides direct numerical evidence that the invariant structure of the layer problem continues to organize the fast dynamics of the full system, even for small but finite values of $\delta$. 

To further illustrate this geometric organization, panels \ref{fig:fast_first_integral}\textbf{(b)} and \ref{fig:fast_first_integral}\textbf{(c)} display the trajectory in the $(S,I,J)$ phase space together with selected level sets of $\Gamma$, corresponding to two distinct exit points associated with the last two fast excursions of the orbit. In both cases, the trajectory closely follows the corresponding level set during the fast phase, confirming that the motion is well approximated by the geometry induced by $\Gamma$. Moreover, the comparison between panels \ref{fig:fast_first_integral}\textbf{(b)} and \ref{fig:fast_first_integral}\textbf{(c)} highlights the robustness of this behavior across different excursions: although the location of the exit point changes, the associated fast dynamics remain organized by the same invariant structure. This confirms that the excursions observed in the full system can be interpreted as fast motions connecting different points of the critical manifold along level sets of $\Gamma$. From a biological perspective, these fast excursions correspond to rapid outbreak events, during which the infected population quickly rises and subsequently declines while the susceptible population undergoes a significant variation. The fact that these outbreak phases occur along trajectories constrained by an underlying invariant quantity indicates that they are not arbitrary transient events, but rather manifestations of a well-defined geometric mechanism governing the fast epidemic dynamics. 

Finally, we emphasize that the above description is derived from a singular analysis, and is therefore rigorously valid in the limit $\delta \to 0$. The numerical results show that this geometric structure persists for small but finite values of $\delta$, and remains remarkably accurate even for $\delta \approx 10^{-3}$, providing strong evidence of the robustness of the fast dynamics with respect to the singular limit. As expected, the agreement gradually deteriorates as $\delta$ increases, reflecting the asymptotic nature of the approximation; nevertheless, for the parameter values considered here, the agreement is excellent.

\section{Discussion}
\label{sec:discussion}

In this work, we investigated the emergence and geometric structure of traveling invasion fronts in a spatial SIRS epidemic model with slow loss of immunity and diffusion acting only on the infected population. Starting from the reaction-diffusion system introduced in Section~\ref{sec:model}, numerical simulations revealed the existence of coherent traveling fronts connecting the EE to the DFE, propagating at an asymptotically constant speed (Figure~\ref{fig:numerical_simulations}). A key observation was the presence of a pronounced separation of scales in the spatial profile: abrupt variations in the infected compartment coexist with much slower changes in the susceptible population. This multiscale structure provided the main motivation for the geometric analysis carried out in Section~\ref{sec:analysis}.

After reducing the PDE model to the traveling-wave system~\eqref{eq:PTW_SIRS_red_sp}, we exploited the small immunity-loss parameter $\delta$ to formulate the problem within the framework of GSPT. For analytical convenience, we reversed the traveling-wave coordinate, thereby inverting the orientation of the orbit. As a consequence, the singular orbit was constructed and analyzed in backward wave coordinates, while still corresponding to a traveling front connecting the EE and DFE in the original spatial variable. Within this framework, the local analysis near the DFE and EE (Section~\ref{sec:lin_stab}) identified the relevant eigendirections and clarified the asymptotic behavior observed at both ends of the front. In particular, Lemma~\ref{lem:local_stab_DFE} characterizes the unstable manifold of the DFE in the reversed formulation, whereas Lemma~\ref{lem:local_stab_EE} shows that the EE possesses a pair of weakly stable complex eigenvalues, explaining the damped oscillations observed numerically in the tail of the traveling profile.

The singular limit $\delta \to 0$ revealed a natural decomposition of the dynamics into fast and slow regimes. The layer problem analyzed in Section~\ref{sec:fast_lim} governs rapid epidemic outbreaks and admits the invariant quantity $\Gamma$, which geometrically organizes the fast excursions away from the critical manifold. Specifically, Proposition~\ref{prop:fast_L} provides a quantitative characterization of the return mechanism generated by these fast dynamics, identifying the smallest admissible value of the susceptible population reached during an outbreak. On the other hand, the reduced problem on the critical manifold describes the gradual flow back of susceptibles due to the slow loss of immunity. The transition between these regimes is regulated by an entry-exit mechanism, whose quantitative description is given in Section \ref{sec:slow_flow}. Taken together, these results yield the singular geometric skeleton of the traveling front described in Section~\ref{sec:singular_full}: a concatenation of fast outbreak episodes, slow recovery phases, and local dynamics near the EE.

The numerical simulations presented in Section~\ref{sec:num_valid} strongly support this geometric interpretation. The decomposition predicted by the singular analysis is clearly visible in the traveling-front profile and in phase space (Figure~\ref{fig:scales}), where the orbit appears as an alternating sequence of slow drifts and fast excursions. Moreover, the numerical entry and exit points agree remarkably well with the predictions obtained from the invariant $\Gamma$ and from the entry-exit relation (Figure~\ref{fig:entryexit_validation}), while the fast portions of the orbit remain closely aligned with the level sets of $\Gamma$ (Figure~\ref{fig:fast_first_integral}). Overall, these computations indicate that the singular geometric structure persists for small but biologically relevant values of $\delta$, well beyond the singular limit in which it was derived.

From a biological perspective, the traveling front can be interpreted as a sequence of epidemic bursts separated by long periods of apparent quiescence. During the fast phases, infection spreads rapidly through the population, generating localized outbreak events. In contrast, during the slow phases, the infected population remains close to zero while immunity gradually wanes and susceptible individuals accumulate. The entry-exit mechanism quantifies precisely how long the system can remain close to this absence of infection manifold before a new outbreak is triggered. Consequently, the slow immunity-loss rate controls not only the temporal separation between epidemic episodes, but also the spatial distance between successive outbreak regions along the traveling front. In this sense, the traveling wave is not merely a spatial propagation phenomenon; rather, it represents the spatial manifestation of the interplay between rapid transmission processes and slow demographic replenishment of susceptible individuals.

It is worth noting that, although the present heterogeneous setting differs substantially from the spatially homogeneous SIRS model analyzed in \cite{jardon2021geometric}, the underlying slow-fast mechanisms remain closely related. Indeed, after introducing the traveling-wave coordinate and reversing its orientation for analytical convenience, the traveling-front problem can be reformulated in a way that allows geometric tools developed for finite-dimensional epidemic systems to be employed in the present context. While the diffusion term determines the existence and propagation of the front, the alternation between fast outbreak excursions and slow recovery phases is dictated by the underlying kinetics, namely the slow loss of immunity encoded by $\delta$. As a consequence, the traveling-wave orbit admits a decomposition into fast outbreak excursions and slow recovery phases, and the invariant-based description of the fast dynamics together with the entry-exit mechanism of Section~\ref{sec:analysis} can be viewed as spatial manifestations of the slow-fast epidemic structure previously identified in homogeneous SIRS models.

Several extensions of the present work appear worth pursuing. A first direction is to enrich the spatial coupling by allowing susceptible individuals to diffuse as well (either through the standard Laplacian, as in \cite{chang2022sparse,wu2021periodic}, or through cross-diffusion, as in \cite{ahmadpoortorkamani2025spatiotemporal,li2025positive}). A second extension concerns the introduction of additional spatial processes acting on faster scales. Motivated by epidemic models with lines of fast diffusion \cite{berestycki2020propagation} or environmental pathogen reservoirs \cite{pang2019sis}, one could couple the present SIRS dynamics to a rapidly spreading component, such as a free-virus compartment or a reservoir host population. This would combine temporal and spatial scale separations, potentially generating richer front geometries and new propagation regimes. Incorporating demographic turnover would introduce a further layer of complexity: if births and deaths occur on a slower timescale than immunity waning, the resulting model would naturally involve three distinct temporal scales and provide a framework in which to study the interplay between epidemic outbreaks, immunity loss, and demographic renewal \cite{kaklamanos2024geometric,kuehn2015multiple}. Another natural extension would be to investigate discontinuous traveling waves and singular heteroclinic connections. The emergence or disappearance of such global structures may induce abrupt transitions between qualitatively different epidemic propagation regimes, leading to sudden reorganizations of the dynamics \cite{li2021shock, Grifo2025II, Grifo2026}. 

\;\\ \\
\noindent\textbf{Acknowledgments.} All the authors are members and acknowledge the support of {\it Gruppo Nazionale di Fisica Matematica} (GNFM) of {\it Istituto Nazionale di Alta Matematica} (INdAM). Rossella Della Marca and Mattia Sensi acknowledge in particular the support of GNFM through the {\it Progetto Giovani GNFM 2025} ``Analisi geometrica della diffusione spaziale di epidemie su più scale temporali'' (CUP E5324001950001). Gabriele Grif\`o acknowledges the INdAM for financial support through a postdoctoral fellowship funded by the European Union – NextGenerationEU (CUP E63C25000470007). Mattia Sensi acknowledges the support of Fondazione Caritro (Cassa di Risparmio di Trento e Rovereto) through the Bando Post-Doc 2024 project ``Modelli matematici di malattie infettive più ospiti e popolazioni eterogenee: applicazioni all’influenza aviaria''.

\appendix

{\footnotesize
	\bibliographystyle{plain}
	\bibliography{biblio}
}

\end{document}